\UseRawInputEncoding
\documentclass[11pt]{article}
\usepackage[T1]{fontenc}
\usepackage{lmodern}
\usepackage{amsmath,amssymb,amsthm,mathtools}
\usepackage[margin=30mm]{geometry}
\usepackage{microtype}
\usepackage{xcolor}
\usepackage[pdfborder={0 0 0}]{hyperref}

\newtheorem{theorem}{Theorem}[section]
\newtheorem{proposition}[theorem]{Proposition}
\newtheorem{lemma}[theorem]{Lemma}
\theoremstyle{remark}
\newtheorem{remark}[theorem]{Remark}
\newtheorem*{remark*}{Remark}
\newtheorem{question}[theorem]{Question}
\numberwithin{equation}{section}
\newcommand{\R}{\mathbb{R}}
\newcommand{\eps}{\varepsilon}
\newcommand{\dd}{\,\mathrm{d}}
\newcommand{\calR}{\mathcal{R}}
\DeclareMathOperator{\Tr}{Tr}
\DeclareMathOperator{\interior}{int}
\DeclareMathOperator{\spanop}{span}
\DeclareMathOperator{\supp}{supp}
\DeclareMathOperator{\diag}{diag}
\DeclarePairedDelimiter{\abs}{\lvert}{\rvert}
\DeclarePairedDelimiter{\norm}{\lVert}{\rVert}
\begin{document}

\title{{Robin eigenvalue ratios: monotonicity on boxes and counterexamples}\thanks{Research supported by NNSF of China (No. 12371110).}}
\author{Guowei Dai\thanks{Corresponding author.
School of Mathematical Sciences, Dalian University of Technology, Dalian, 116024, P.R. China.
E-mail: daiguowei@dlut.edu.cn.},\,\,\, Yingxin Sun\thanks{School of Mathematical Sciences, Dalian University of Technology, Dalian, 116024, P.R. China. E-mail: sunyingxin2023@mail.dlut.edu.cn}\\
}
\date{}
\maketitle

\renewcommand{\abstractname}{Abstract}
\begin{abstract}
We disprove Laugesen's conjecture that the ratio of the first two
Robin eigenvalues decreases with the positive boundary parameter.
On any prescribed compact interval of positive parameters, the ratio
is strictly increasing on a suitable connected polyhedral domain in
every dimension at least two; in the plane, the domain can be a
simply connected polygon. In contrast, we prove strict
decrease on every rectangular box.
Both results follow from a comparison of the relative growth of
interval eigenvalues.
\end{abstract}



\section{Introduction and main results}

\quad \, Let $\Omega\subset\R^n$ be a bounded connected Lipschitz domain.
We consider the following Robin eigenvalue problem
\begin{equation*}
\begin{cases}
 -\Delta u=\lambda u\,\,\,&\text{in }\Omega,\\
 \partial_\nu u+\alpha u=0&\text{on }\partial\Omega,
 \end{cases}
\end{equation*}
where $\nu$ is the outward unit normal. The associated quadratic form is
\[
\mathfrak a_{\Omega,\alpha}[u]
 =\int_\Omega \abs{\nabla u}^2\dd x
   +\alpha\int_{\partial\Omega}\abs[\big]{\Tr u}^2\dd S,
 \,\,\, u\in H^1(\Omega).
\]
Here $\Tr u$ denotes the boundary trace of $u$.
We write
$\abs{\Omega}$ for volume and $\abs{\partial\Omega}_{n-1}$ for
boundary measure.
For $\alpha>0$, the eigenvalues, repeated according to multiplicity,
satisfy
\[
 0<\lambda_1(\Omega;\alpha)<\lambda_2(\Omega;\alpha)\le\cdots.
\]
We study the ratio
\[
 \calR_\Omega(\alpha)
 =\frac{\lambda_2(\Omega;\alpha)}{\lambda_1(\Omega;\alpha)},
 \,\,\, \alpha\ne0.
\]
Our main question concerns $\alpha>0$. The distinction between this
case and $\alpha<0$ is explained below; the ratio is undefined at
$\alpha=0$.
For a disconnected open set with finitely many Lipschitz components,
the eigenvalues are obtained by merging the component spectra and
arranging them in nondecreasing order, with multiplicities.

The ratio of the first two eigenvalues measures the separation of
the first two spectral levels relative to the first. Its study goes
back to Payne, P\'olya and Weinberger~\cite{PPW} for the Dirichlet
Laplacian. Their conjecture that the disk maximizes this ratio was
proved, in all dimensions with the ball in place of the disk, by
Ashbaugh and Benguria~\cite{AshbaughBenguria}. For the Robin problem,
Payne and Schaefer~\cite{PayneSchaefer} obtained a planar bound of
three for a range of positive boundary parameters and raised a
corresponding shape optimization question. These problems compare
different domains while the boundary condition is fixed. Our question
is different: the domain is fixed, and only $\alpha$ varies.

The Robin parameter connects the Neumann and Dirichlet problems.
On a fixed connected Lipschitz domain, as $\alpha\downarrow0$ the
first Robin eigenvalue tends to zero and the second tends to the
first positive Neumann eigenvalue. As $\alpha\to+\infty$, they tend
to their Dirichlet counterparts; see, for example,~\cite{Laugesen}.
Thus $\calR_\Omega$ starts at $+\infty$ and has a finite Dirichlet
limit. This makes decrease plausible, but the endpoint limits do
not determine what happens at intermediate parameters. Moreover,
the increase of the individual eigenvalues does not determine the
monotonicity of their ratio: their relative rates of increase must
be compared.

Laugesen~\cite[Conjecture B]{Laugesen} \emph{conjectured that
$\calR_\Omega$ is decreasing on $(0,+\infty)$ for every bounded
Lipschitz domain, and explicitly left the rectangle case unresolved}.
He proved that $\alpha\calR_\Omega(\alpha)$ is increasing for positive
$\alpha$. The increase of the fundamental gap
$\lambda_2-\lambda_1$ was studied by Smits~\cite[Section~4]{Smits}
for positive parameters and proved by
Laugesen~\cite[Theorem~2.1]{Laugesen} on boxes for all real parameters.
Neither of these monotonicity statements decides the ratio for
$\alpha>0$: two increasing positive functions need not have a
decreasing quotient. {The decrease on balls was proved in our separate
preprint~\cite[Theorem~1.3]{DaiSun}.} Freitas and Kennedy~\cite{FreitasKennedy}
studied eigenvalue bounds and shape optimization on rectangles and
their disjoint unions, with the boundary parameter fixed. Here we compare the responses of different spectral levels to a
change in that parameter.

Our first result shows that the conjecture fails.
\begin{theorem}\label{thm:counterexample}
Let $n\ge2$ and let $0<\alpha_-<\alpha_+$ be prescribed. There exists
a bounded connected Lipschitz polyhedral domain $\Omega\subset\R^n$
such that $\calR_\Omega$ is continuously differentiable near
$[\alpha_-,\alpha_+]$ and satisfies $\calR_\Omega'(\alpha)>0$
throughout this interval. In particular,
\begin{equation}\label{eq:reversal}
 \calR_\Omega(\alpha_+)>\calR_\Omega(\alpha_-).
\end{equation}
When $n=2$, $\Omega$ can be chosen to be a simply connected polygon.
\end{theorem}

The domain depends on the prescribed interval. Its ratio is monotone
in neither direction on $(0,+\infty)$: it increases throughout that
interval but tends to $+\infty$ as $\alpha\downarrow0$; see
Remark~\ref{rem:nonmonotone}.
Our second result proves the conjecture for every rectangular box.
The same parameter $\alpha$ is imposed on all faces.

\begin{theorem}\label{thm:boxes}
Let $n\ge1$ and $B=\prod_{j=1}^n(-t_j,t_j)$ with $t_j>0$.
Then $\calR_B$ is differentiable on $(0,+\infty)$ and $\calR_B'(\alpha)<0$.
Writing $T=\max_{1\le j\le n}t_j$, its endpoint limits are
\begin{equation}\label{eq:boxlimits}
 \lim_{\alpha\downarrow0}\calR_B(\alpha)=+\infty,
 \,\,\,
 \lim_{\alpha\to+\infty}\calR_B(\alpha)
 =1+\frac{3/T^2}{\sum_{j=1}^n t_j^{-2}}.
\end{equation}
The conclusion includes boxes whose longest edge is not unique.
\end{theorem}

For negative parameters, the box ratio is also strictly decreasing,
from $1$ at $-\infty$ to $-\infty$ at $0^-$; see
Proposition~\ref{prop:negative-boxes}. This follows from Laugesen's
gap theorem, whereas positive parameters require the additional
relative-growth comparison below. The end of
Section~\ref{sec:counterexample} explains the different effect of
negative parameters on the shrinking passage.

The box proof requires a comparison of relative growth rates.
Writing $L=\lambda_1$ and $G=\lambda_2-\lambda_1$, we have
\begin{equation}\label{eq:relativecriterion}
 \calR'=\frac{G}{L}\left(\frac{G'}{G}-\frac{L'}{L}\right).
\end{equation}
Thus, for $\alpha>0$, we need $G'/G<L'/L$; knowing $G'>0$ is not
enough. The difficulty is that $L$ combines interval eigenvalues at
different scales. Let $p(a)$ and $q(a)$ be the first two eigenvalues
of $(-1,1)$ for $a>0$, put $d=q-p$, and define
\[
 E(a)=\frac{ap'(a)}{p(a)}.
\]
We prove that $E$ is strictly decreasing and $ad'(a)/d(a)<E(a)$.
For a box, put $L_j=t_j^{-2}p(\alpha t_j)$ and $T=\max_j t_j$.
Then $L=\sum_jL_j$, and the gap is $G=T^{-2}d(\alpha T)$
\cite[Corollary~6.3]{Laugesen}. Since $t_j\le T$ and the positive
weights $L_j/L$ sum to one,
\[
 \alpha\frac{G'}{G}
 <E(\alpha T)
 \le\sum_{j=1}^n\frac{L_j}{L}E(\alpha t_j)
 =\alpha\frac{L'}{L}.
\]
The decrease of $E$ handles the different side lengths, while the
first strict inequality gives $\calR'<0$.

For the counterexample, take two cubes of side lengths $2$ and
$2s$, with $s<1$. Their first-eigenvalue ratio satisfies
\[
 R_s(\alpha)=s^{-2}\frac{p(s\alpha)}{p(\alpha)},
 \,\,\,
 \alpha\frac{R_s'(\alpha)}{R_s(\alpha)}
 =E(s\alpha)-E(\alpha)>0.
\]
Choosing $s>\sqrt{n/(n+3)}$ ensures that the smaller cube's first
eigenvalue lies below the larger cube's second eigenvalue. Thus the
two first eigenvalues are the first two eigenvalues of the union,
and their ratio increases.
We join the cubes by a thin passage. For fixed positive $\alpha$,
the lowest eigenvalue of the separated passage tends to $+\infty$.
Matching variational bounds then give convergence to the union's
spectrum. To preserve strict increase on a whole compact interval,
we use the concavity of $\lambda_1$ and $\lambda_1+\lambda_2$ to
obtain uniform convergence of their parameter derivatives.

The same decreasing function $E$ therefore explains both results:
on boxes it compares the gap with a sum of first eigenvalues;
on the two cubes it compares first eigenvalues at different scales.
The interval formulas and domain perturbation methods are
standard~\cite{Laugesen,DancerDaners,Daners}. The relative-growth
comparison and the choice of unequal components turn these tools
into the box monotonicity theorem and connected counterexamples.

The rest of this paper is arranged as follows. Section~\ref{sec:interval} gives the interval calculations.
Section~\ref{sec:boxes} proves the box theorem and discusses negative
parameters on boxes. Section~\ref{sec:counterexample} constructs
the connected counterexamples and proves their spectral convergence;
it also explains why the passage argument changes for negative
parameters and ends with the convex-domain question.

\section{Interval eigenvalues and logarithmic derivatives}
\label{sec:interval}

\quad\, We keep track of two separate comparisons: how the first eigenvalue
responds at different interval lengths, and how the gap responds
relative to the first eigenvalue on a single interval.

For $a>0$, let $x(a)\in(0,\pi/2)$ and $y(a)\in(\pi/2,\pi)$ be the
unique solutions of
\begin{equation}\label{eq:roots}
 x\tan x=a,\,\,\, -y\cot y=a.
\end{equation}
The functions $x\mapsto x\tan x$ and $y\mapsto-y\cot y$ are strictly
increasing from zero to $+\infty$ on the indicated intervals.
Their derivatives are positive, so their inverses are smooth.
Set
\[
 p(a)=x(a)^2,\,\,\, q(a)=y(a)^2,\,\,\,
 c(a)=a+a^2,\,\,\, d(a)=q(a)-p(a).
\]
On $I_t=(-t,t)$, the first eigenfunction is proportional to
$\cos\bigl(x(\alpha t)z/t\bigr)$ and the second to
$\sin\bigl(y(\alpha t)z/t\bigr)$.
The boundary condition at $z=t$ gives the two equations in
\eqref{eq:roots}, with $a=\alpha t$, and hence
\[
 \lambda_1(I_t;\alpha)=t^{-2}p(\alpha t),\,\,\,
 \lambda_2(I_t;\alpha)=t^{-2}q(\alpha t).
\]
These standard interval formulas are also given in~\cite{Laugesen}.

\begin{lemma}\label{lem:derivatives}
For $a>0$,
\begin{equation}\label{eq:derivatives}
 p'(a)=\frac{2p(a)}{p(a)+c(a)},\,\,\,
 q'(a)=\frac{2q(a)}{q(a)+c(a)}.
\end{equation}
The function
\begin{equation}\label{eq:elasticity}
 E(a)=\frac{ap'(a)}{p(a)}
 =\frac{2}{1+2x(a)/\sin\bigl(2x(a)\bigr)}
\end{equation}
is strictly decreasing and satisfies $0<E(a)<1$. Moreover,
\begin{equation}\label{eq:gapelasticity}
 \frac{ad'(a)}{d(a)}
 =E(a)\frac{c(a)}{q(a)+c(a)}<E(a),
\end{equation}
and $q(a)/p(a)$ is strictly decreasing, with
\begin{equation}\label{eq:qpbound}
 \frac{q(a)}{p(a)}>4,
 \,\,\, \lim_{a\to+\infty}\frac{q(a)}{p(a)}=4.
\end{equation}
\end{lemma}

\begin{proof}
In the following calculations, $x=x(a)$, $y=y(a)$, $p=p(a)$,
$q=q(a)$ and $c=c(a)$.
Using $\tan x=a/x$ and $\sec^2x=1+\tan^2x$, we obtain
\[
 \frac{\mathrm da}{\mathrm dx}
 =\tan x+x\sec^2x
 =\frac{a}{x}+x+\frac{a^2}{x}
 =\frac{p+c}{x}.
\]
Similarly, $\cot y=-a/y$ gives
\[
 \frac{\mathrm da}{\mathrm dy}
 =-\cot y+y\csc^2y
 =\frac{a}{y}+y+\frac{a^2}{y}
 =\frac{q+c}{y}.
\]
Since $p'=2x x'$ and $q'=2y y'$, these identities prove
\eqref{eq:derivatives}. Substituting $a=x\tan x$ then yields
\[
 E(a)=\frac{2a}{p+c}
 =\frac{2}{1+a+p/a}
 =\frac{2}{1+x\tan x+x\cot x}
 =\frac{2}{1+2x/\sin(2x)}.
\]
For $0<u<\pi$,
\[
 \frac{\mathrm d}{\mathrm du}\left(\frac{u}{\sin u}\right)
 =\frac{\sin u-u\cos u}{\sin^2u},
 \,\,\,
 \sin u-u\cos u=\int_0^u v\sin v\dd v>0.
\]
Thus $u/\sin u$ is strictly increasing. Since $x(a)$ is increasing,
\eqref{eq:elasticity} shows that $E$ is strictly decreasing.
The inequalities $0<\sin(2x)<2x$ give $0<E<1$.
In particular, for $0<t<T$ and fixed $\alpha>0$,
\[
 \alpha\frac{\partial}{\partial\alpha}
       \log\lambda_1(I_t;\alpha)
 =E(\alpha t)>E(\alpha T)
 =\alpha\frac{\partial}{\partial\alpha}
       \log\lambda_1(I_T;\alpha).
\]
This is the comparison between different scales.

Subtracting the two identities in \eqref{eq:derivatives}, we obtain
\[
 \begin{aligned}
 d'
 &=\frac{2q}{q+c}-\frac{2p}{p+c}\\
 &=\frac{2q(p+c)-2p(q+c)}{(p+c)(q+c)}\\
 &=\frac{2c(q-p)}{(p+c)(q+c)}.
 \end{aligned}
\]
Dividing by $d=q-p>0$ and multiplying by $a$ gives
\[
 \frac{ad'}{d}
 =\frac{2a}{p+c}\,\frac{c}{q+c}
 =E(a)\frac{c}{q+c}<E(a),
\]
which proves \eqref{eq:gapelasticity}. The factor $c/(q+c)$ lies
strictly between zero and one. It is precisely the strict loss in
relative growth that will make the box argument work. Also,
\[
 \frac{\mathrm d}{\mathrm da}\log\left(\frac{q}{p}\right)
 =\frac{q'}q-\frac{p'}p
 =\frac{2}{q+c}-\frac{2}{p+c}
 =-\frac{2(q-p)}{(p+c)(q+c)}<0.
\]
Finally, \eqref{eq:roots} implies $x(a)\to\pi/2$ and
$y(a)\to\pi$ as $a\to+\infty$. Hence $q(a)/p(a)\to4$.
A strictly decreasing function with this limit is strictly greater
than four at every finite positive $a$, proving \eqref{eq:qpbound}.
\end{proof}

We will also use the limits
\begin{equation}\label{eq:plimits}
 \begin{gathered}
 p(a)\sim a\quad(a\downarrow0),\,\,\,
 q(a)\longrightarrow\frac{\pi^2}{4}\quad(a\downarrow0),\\
 p(a)\longrightarrow\frac{\pi^2}{4},\,\,\,
 q(a)\longrightarrow\pi^2\quad(a\to+\infty).
 \end{gathered}
\end{equation}
Indeed, $p(a)/a=x/\tan x\to1$ as $a\downarrow0$, while the
remaining limits follow from the endpoints in \eqref{eq:roots}.

\section{Monotonicity on rectangular boxes}
\label{sec:boxes}

\subsection{Positive parameters}

\quad\, By separation of variables, each eigenvalue of a box is a
sum of interval eigenvalues. The box gap is therefore the smallest
interval gap among its edges. The next lemma shows that this minimum
is attained on a longest edge, recovering the formula in~\cite{Laugesen}.

\begin{lemma}\label{lem:longest}
Fix $\alpha>0$. The interval gap
\[
 \delta_\alpha(t)=t^{-2}d(\alpha t)
\]
is strictly decreasing for $t>0$. Consequently, for the box in
Theorem~\ref{thm:boxes},
\begin{equation}\label{eq:boxgap}
 \begin{aligned}
 \lambda_1(B;\alpha)=\sum_{j=1}^n t_j^{-2}p(\alpha t_j),\,\,\,
 \lambda_2(B;\alpha)-\lambda_1(B;\alpha)=T^{-2}d(\alpha T).
 \end{aligned}
\end{equation}
\end{lemma}

\begin{proof}
Differentiation and \eqref{eq:gapelasticity} give
\[
 \frac{t\delta_\alpha'(t)}{\delta_\alpha(t)}
 =-2+\frac{\alpha t\,d'(\alpha t)}{d(\alpha t)}
 <-2+E(\alpha t)<-1.
\]
In particular, $\delta_\alpha'(t)<0$.
By separation of variables, the eigenvalues of $B$ are all sums
\[
 \sum_{j=1}^n\lambda_{m_j}(I_{t_j};\alpha),
 \,\,\, m_j\ge1.
\]
Taking every $m_j=1$ gives the first eigenvalue. In any other sum,
at least one coordinate contributes at least its second eigenvalue.
The smallest possible increase is attained by changing only one
coordinate from its first to its second eigenvalue. Therefore
\[
 \lambda_2(B;\alpha)-\lambda_1(B;\alpha)
 =\min_{1\le j\le n}\delta_\alpha(t_j)
 =\delta_\alpha(T),
\]
where the final equality uses the strict decrease just proved.
This establishes \eqref{eq:boxgap}.
\end{proof}

Now we present the argument of Theorem \ref{thm:boxes}.

\begin{proof}[Proof of Theorem~\ref{thm:boxes}]
By Lemma~\ref{lem:longest}, the gap has a single expression determined
by $T$, independently of which longest edge is selected. Write
\[
 L_j(\alpha)=t_j^{-2}p(\alpha t_j),\,\,\,
 L(\alpha)=\sum_{j=1}^n L_j(\alpha),\,\,\,
 G(\alpha)=T^{-2}d(\alpha T).
\]
All these functions are positive and smooth for $\alpha>0$.
Since $L_j'=t_j^{-1}p'(\alpha t_j)$, we have
\[
 \alpha\frac{L_j'(\alpha)}{L_j(\alpha)}
 =\frac{\alpha t_j p'(\alpha t_j)}{p(\alpha t_j)}
 =E(\alpha t_j).
\]
The numbers $L_j/L$ are positive and sum to one. The decrease of
$E$, together with $t_j\le T$, therefore gives
\begin{equation}\label{eq:weighted}
 \alpha\frac{L'(\alpha)}{L(\alpha)}
 =\sum_{j=1}^n\frac{L_j(\alpha)}{L(\alpha)}E(\alpha t_j)
 \ge E(\alpha T).
\end{equation}
On the other hand, \eqref{eq:gapelasticity} gives
\begin{equation}\label{eq:strictgap}
 \alpha\frac{G'(\alpha)}{G(\alpha)}
 =E(\alpha T)\frac{c(\alpha T)}{q(\alpha T)+c(\alpha T)}
 <E(\alpha T).
\end{equation}
Subtracting \eqref{eq:weighted} from \eqref{eq:strictgap}, and using
$\alpha>0$, yields
\[
 \frac{\mathrm d}{\mathrm d\alpha}
 \log\left(\frac{G(\alpha)}{L(\alpha)}\right)
 =\frac{G'(\alpha)}{G(\alpha)}
  -\frac{L'(\alpha)}{L(\alpha)}<0.
\]
Using \eqref{eq:relativecriterion}, it follows that
\[
 \calR_B'(\alpha)
 =\frac{G(\alpha)}{L(\alpha)}
  \left(\frac{G'(\alpha)}{G(\alpha)}
       -\frac{L'(\alpha)}{L(\alpha)}\right)<0.
\]
If several edges have length $2T$, each gives the same smooth
expression $T^{-2}d(\alpha T)$ for the gap. Thus a possible
multiplicity of $\lambda_2$ causes no loss of differentiability.
The strict inequality \eqref{eq:strictgap} also applies when every
edge has the same length.

As $\alpha\downarrow0$, \eqref{eq:plimits} gives
\[
 L(\alpha)\sim\alpha\sum_{j=1}^n t_j^{-1},\,\,\,
 G(\alpha)\longrightarrow\frac{\pi^2}{4T^2}.
\]
Hence $G/L\to+\infty$. As $\alpha\to+\infty$,
\[
 L(\alpha)\longrightarrow\frac{\pi^2}{4}\sum_{j=1}^n t_j^{-2},
 \,\,\, G(\alpha)\longrightarrow\frac{3\pi^2}{4T^2}.
\]
Substitution into $\calR_B=1+G/L$ proves \eqref{eq:boxlimits}.
\end{proof}

\begin{remark}
The same comparison gives the explicit estimate
\[
 \frac{\mathrm d}{\mathrm d\alpha}
 \log\bigl(\calR_B(\alpha)-1\bigr)
 \le-\frac{E(\alpha T)}{\alpha}
       \frac{q(\alpha T)}{q(\alpha T)+c(\alpha T)}<0.
\]
Thus it is the relative growth of the gap, rather than the sign of
its derivative alone, that determines the decrease of the ratio.
\end{remark}

\subsection{Negative parameters on boxes}

\quad\, The sign of the first eigenvalue changes the comparison. For
$\alpha<0$, an increasing gap is sufficient for a decreasing signed
ratio, unlike the positive-parameter case. The following consequence
of Laugesen's gap theorem is included to make this distinction
explicit.

\begin{proposition}
\label{prop:negative-boxes}
For the box $B$ in Theorem~\ref{thm:boxes}, the function
$\calR_B$ is differentiable on $(-\infty,0)$ and satisfies
\[
 \calR_B'(\alpha)<0\,\,\,(\alpha<0),
 \,\,\,
 \lim_{\alpha\to-\infty}\calR_B(\alpha)=1,
 \,\,\,
 \lim_{\alpha\uparrow0}\calR_B(\alpha)=-\infty.
\]
There is a unique $\alpha_0<0$ such that
$\lambda_2(B;\alpha_0)=0$. The ratio belongs to $(0,1)$ for
$\alpha<\alpha_0$ and is negative for $\alpha_0<\alpha<0$.
\end{proposition}

\begin{proof}
Write $L(\alpha)=\lambda_1(B;\alpha)$ and
$G(\alpha)=\lambda_2(B;\alpha)-\lambda_1(B;\alpha)$.
The longest-edge formula holds for all real parameters
\cite[Corollary~6.3]{Laugesen}, so the smooth dependence of interval
eigenvalues gives smooth functions $L$ and $G$, including at
parameters where $\lambda_2(B;\alpha)=0$. The gap theorem
\cite[Theorem~2.1]{Laugesen} gives $G'\ge0$ and
$G(\alpha)\to0$ as $\alpha\to-\infty$.

The constant trial function gives
\[
 L(\alpha)\le
 \alpha\frac{\abs{\partial B}_{n-1}}{\abs{B}}<0.
\]
Let $u_1(\alpha)$ be a real, $L^2(B)$-normalized first eigenfunction.
Differentiating the form along $u_1$, and using its eigenfunction
equation and the derivative of its normalization, yields
\[
 \begin{aligned}
 L'(\alpha)
 &=\int_{\partial B}\abs[\big]{\Tr u_1}^2\dd S
   +2\mathfrak a_{B,\alpha}(u_1,\partial_\alpha u_1)\\
 &=\int_{\partial B}\abs[\big]{\Tr u_1}^2\dd S
   +2L\langle u_1,\partial_\alpha u_1\rangle_{L^2(B)}\\
 &=\int_{\partial B}\abs[\big]{\Tr u_1}^2\dd S>0.
 \end{aligned}
\]
{Indeed, a zero trace would imply
$L=\int_B\abs{\nabla u_1}^2\dd x\ge0$, contrary to $L<0$.}
As $L<0$ and $G>0$, we now obtain
\[
 \calR_B'(\alpha)
 =\frac{G'L-GL'}{L^2}<0.
\]
Here $G'L\le0$ and $-GL'<0$; no relative-growth estimate is needed.

The constant-function bound implies $L\to-\infty$ as
$\alpha\to-\infty$. Together with $G\to0$, it gives
$\calR_B=1+G/L\to1$. At $\alpha\uparrow0$, continuity at the
Neumann problem gives $L\to0^-$ and $G\to\pi^2/(4T^2)>0$, so
$\calR_B\to-\infty$. The strict decrease and these two limits give
a unique zero of the ratio. Since $L<0$ throughout this range, it
is exactly the unique zero of $\lambda_2$. The remaining signs
follow as well.
\end{proof}

\section{Connected counterexamples}
\label{sec:counterexample}

\quad\, We first choose the components so that their two first eigenvalues
are the first two eigenvalues of the union. This must be checked
before comparing their ratio: a cube that is too small would
contribute its first eigenvalue only at a higher spectral position.

Fix $n\ge2$ and choose
\begin{equation}\label{eq:srange}
 \sqrt{\frac{n}{n+3}}<s<1.
\end{equation}
Let
\[
 \begin{gathered}
 Q_1=(-1,1)^n,\,\,\,
 Q_2=(2,2+2s)\times(-s,s)^{n-1},\,\,\,
 D=Q_1\cup Q_2.
 \end{gathered}
\]
The cubes have side lengths $2$ and $2s$, and are separated by a
distance of one in the first coordinate. Separation of variables
gives
\[
 \lambda_1(Q_1;\alpha)=n p(\alpha),\,\,\,
 \lambda_1(Q_2;\alpha)=n s^{-2}p(s\alpha).
\]
Their first-eigenvalue ratio is therefore
\begin{equation}\label{eq:Rlimit}
 R_s(\alpha)
 =\frac{\lambda_1(Q_2;\alpha)}{\lambda_1(Q_1;\alpha)}
 =s^{-2}\frac{p(s\alpha)}{p(\alpha)}.
\end{equation}

\begin{lemma}\label{lem:disconnected}
{The function $R_s$ satisfies $R_s'(\alpha)>0$ for every $\alpha>0$, and}
\begin{equation}\label{eq:Rrange}
 \frac1s<R_s(\alpha)<\frac1{s^2}.
\end{equation}
For every $\alpha>0$,
\begin{equation}\label{eq:ordering}
 \lambda_1(D;\alpha)=\lambda_1(Q_1;\alpha),\,\,\,
 \lambda_2(D;\alpha)=\lambda_1(Q_2;\alpha).
\end{equation}
In particular, $\calR_D=R_s$ is strictly increasing.
\end{lemma}

\begin{proof}
For a single box we compared the gap with its first eigenvalue.
Here both quantities being compared are first eigenvalues, but at
different scales. Taking the logarithmic derivative of
\eqref{eq:Rlimit} and using the chain rule, we find
\[
 \begin{aligned}
 \alpha\frac{R_s'(\alpha)}{R_s(\alpha)}
 &=\alpha\left(\frac{s p'(s\alpha)}{p(s\alpha)}
              -\frac{p'(\alpha)}{p(\alpha)}\right)\\
 &=\frac{s\alpha p'(s\alpha)}{p(s\alpha)}
   -\frac{\alpha p'(\alpha)}{p(\alpha)}\\
 &=E(s\alpha)-E(\alpha)\\
 &>0.
 \end{aligned}
\]
The last inequality follows from $s\alpha<\alpha$ and the strict
decrease of $E$ in Lemma~\ref{lem:derivatives}. Thus $R_s$ is strictly
increasing. To compute its limit at zero, write
\[
 R_s(\alpha)
 =\frac1s\,
   \frac{p(s\alpha)/(s\alpha)}{p(\alpha)/\alpha}.
\]
Since both fractions in the final quotient tend to one,
\[
 \lim_{\alpha\downarrow0}R_s(\alpha)=\frac1s.
\]
As $\alpha\to+\infty$, both $p(s\alpha)$ and $p(\alpha)$ tend to
$\pi^2/4$, and hence
\[
 \lim_{\alpha\to+\infty}R_s(\alpha)
 =\frac1{s^2}\,\frac{\pi^2/4}{\pi^2/4}
 =\frac1{s^2}.
\]
The strict increase gives the strict bounds in \eqref{eq:Rrange}.

It remains to identify the first two eigenvalues of the union.
For $Q_1$, the second eigenvalue is obtained by choosing the second
interval eigenvalue in one coordinate, so
\[
 \lambda_2(Q_1;\alpha)=(n-1)p(\alpha)+q(\alpha).
\]
Consequently,
\[
 \begin{aligned}
 \frac{\lambda_2(Q_1;\alpha)}{\lambda_1(Q_1;\alpha)}
 &=\frac{(n-1)p(\alpha)+q(\alpha)}{n p(\alpha)}\\
 &=1+\frac1n\left(\frac{q(\alpha)}{p(\alpha)}-1\right)\\
 &>1+\frac3n,
 \end{aligned}
\]
where \eqref{eq:qpbound} was used. The choice \eqref{eq:srange}
means that $s^2>n/(n+3)$, or equivalently
\[
 \frac1{s^2}<\frac{n+3}{n}=1+\frac3n.
\]
Combining these inequalities gives the full ordering
\[
 1<\frac1s<R_s(\alpha)<\frac1{s^2}<1+\frac3n
 <\frac{\lambda_2(Q_1;\alpha)}{\lambda_1(Q_1;\alpha)}.
\]
Multiplication by $\lambda_1(Q_1;\alpha)>0$ yields
\[
 \lambda_1(Q_1;\alpha)<\lambda_1(Q_2;\alpha)
 <\lambda_2(Q_1;\alpha).
\]
Every other eigenvalue of $Q_1$ is at least $\lambda_2(Q_1;\alpha)$,
and every other eigenvalue of $Q_2$ is strictly greater than
$\lambda_1(Q_2;\alpha)$, because its first eigenvalue is simple.
Thus the first two entries in the merged spectrum are exactly those
in \eqref{eq:ordering}. If $s=1$, the two first
eigenvalues would coincide and their ratio would be the constant
one, so the unequal sizes are essential.
\end{proof}

For the lower bound on the connected domain, we will remove the
Robin boundary term on the small attachment faces of the cubes.
The following lemma shows that this change has a vanishing effect
on each fixed eigenvalue. We write $\abs{\cdot}_{n-1}$ for surface
measure and use the norm
\[
 \norm{u}_{H^1(Q)}^2
 =\norm{u}_{L^2(Q)}^2+\norm{\nabla u}_{L^2(Q)}^2.
\]

\begin{lemma}\label{lem:window}
Let $Q\subset\R^n$ be a fixed bounded Lipschitz domain, let $\alpha>0$,
and let $\Gamma_\eps\subset\partial Q$ be measurable sets such that
$\abs{\Gamma_\eps}_{n-1}\to0$ as $\eps\downarrow0$.
Let $\nu_{k,\eps}$ be the eigenvalues of the form
\[
 \mathfrak a_\eps[u]=\int_Q\abs{\nabla u}^2\dd x
 +\alpha\int_{\partial Q\setminus\Gamma_\eps}
       \abs[\big]{\Tr u}^2\dd S,
 \,\,\, u\in H^1(Q).
\]
Then $\nu_{k,\eps}\to\lambda_k(Q;\alpha)$ for each fixed $k$.
\end{lemma}

\begin{proof}
The compactness of the trace map from $H^1(Q)$ to $L^2(\partial Q)$
implies
\begin{equation}\label{eq:eta}
 \eta_\eps:=\sup_{\norm{u}_{H^1(Q)}\le1}
       \int_{\Gamma_\eps}\abs[\big]{\Tr u}^2\dd S
       \longrightarrow0.
\end{equation}
To see this, suppose otherwise. There would be a sequence
$\eps_j\downarrow0$, a number $b>0$, and functions $u_j$ with
$\norm{u_j}_{H^1(Q)}\le1$ such that
\[
 \int_{\Gamma_{\eps_j}}\abs[\big]{\Tr u_j}^2\dd S\ge b.
\]
After passage to a subsequence, the traces converge strongly in
$L^2(\partial Q)$ to some $v$. But
\[
 \begin{aligned}
 \int_{\Gamma_{\eps_j}}\abs[\big]{\Tr u_j}^2\dd S
 &\le 2\int_{\Gamma_{\eps_j}}\abs[\big]{\Tr u_j-v}^2\dd S
       +2\int_{\Gamma_{\eps_j}}\abs{v}^2\dd S\\
 &\le2\norm[\big]{\Tr u_j-v}_{L^2(\partial Q)}^2
       +2\int_{\Gamma_{\eps_j}}\abs{v}^2\dd S
 \longrightarrow0.
 \end{aligned}
\]
Here the first term tends to zero by strong convergence, and the
second does so by absolute continuity of the integral of
$\abs{v}^2$, since $\abs{\Gamma_{\eps_j}}_{n-1}\to0$.
This contradiction proves \eqref{eq:eta}.

By homogeneity, \eqref{eq:eta} gives, for every $u\in H^1(Q)$,
\[
 \int_{\Gamma_\eps}\abs[\big]{\Tr u}^2\dd S
 \le\eta_\eps\norm{u}_{H^1(Q)}^2.
\]
The boundary term in $\mathfrak a_\eps$ is nonnegative because
$\alpha>0$, so $\norm{\nabla u}_{L^2(Q)}^2\le\mathfrak a_\eps[u]$.
We can therefore compare the two forms step by step:
\[
 \begin{aligned}
 \mathfrak a_\eps[u]
 &\le\mathfrak a_{Q,\alpha}[u]\\
 &=\mathfrak a_\eps[u]
   +\alpha\int_{\Gamma_\eps}\abs[\big]{\Tr u}^2\dd S\\
 &\le\mathfrak a_\eps[u]
    +\alpha\eta_\eps\left(
       \norm{\nabla u}_{L^2(Q)}^2+\norm{u}_{L^2(Q)}^2\right)\\
 &\le(1+\alpha\eta_\eps)\mathfrak a_\eps[u]
       +\alpha\eta_\eps\norm{u}_{L^2(Q)}^2.
 \end{aligned}
\]
For $u\ne0$, division by $\norm{u}_{L^2(Q)}^2$ gives the same
comparison for the Rayleigh quotients. The min--max formula is
\[
 \nu_{k,\eps}
 =\min_{\substack{V\subset H^1(Q)\\ \dim V=k}}
   \;\max_{u\in V\setminus\{0\}}
      \frac{\mathfrak a_\eps[u]}{\norm{u}_{L^2(Q)}^2},
\]
and the corresponding formula for $\lambda_k(Q;\alpha)$ has the
same class of subspaces. Taking these maxima and minima in the
quotient comparison yields
\[
 \nu_{k,\eps}\le\lambda_k(Q;\alpha)
 \le(1+\alpha\eta_\eps)\nu_{k,\eps}+\alpha\eta_\eps.
\]
Equivalently,
\[
 \frac{\lambda_k(Q;\alpha)-\alpha\eta_\eps}
      {1+\alpha\eta_\eps}
 \le\nu_{k,\eps}\le\lambda_k(Q;\alpha).
\]
Both bounds tend to $\lambda_k(Q;\alpha)$, proving the claim.
\end{proof}

For $0<\eps<s$, set
\begin{equation}\label{eq:domain}
 \begin{gathered}
 T_\eps=(1,2)\times(-\eps,\eps)^{n-1},\\
 \Omega_\eps=\interior\Bigl(
 \overline{Q_1}\cup\overline{T_\eps}\cup\overline{Q_2}\Bigr).
 \end{gathered}
\end{equation}
These are bounded connected Lipschitz polyhedral domains. In the
plane they are simply connected polygons; see Figure~\ref{fig:domain}.

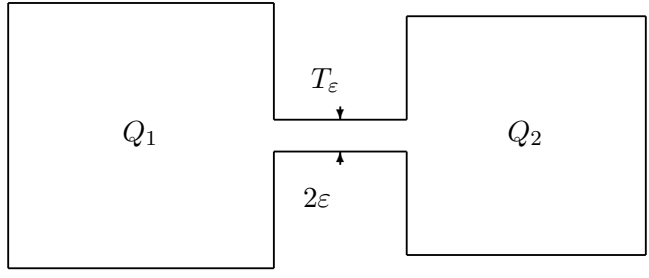
\begin{figure}[ht]
\centering
\setlength{\unitlength}{1pt}
\begin{picture}(280,118)(-20,-8)
\linethickness{0.7pt}
\put(0,0){\line(1,0){100}}
\put(100,0){\line(0,1){44}}
\put(100,44){\line(1,0){50}}
\put(150,44){\line(0,-1){39}}
\put(150,5){\line(1,0){90}}
\put(240,5){\line(0,1){90}}
\put(240,95){\line(-1,0){90}}
\put(150,95){\line(0,-1){39}}
\put(150,56){\line(-1,0){50}}
\put(100,56){\line(0,1){44}}
\put(100,100){\line(-1,0){100}}
\put(0,100){\line(0,-1){100}}
\put(43,48){$Q_1$}
\put(188,48){$Q_2$}
\put(114,68){$T_\eps$}
\put(111,23){$2\eps$}
\put(125,39){\vector(0,1){5}}
\put(125,61){\vector(0,-1){5}}
\end{picture}
\caption{Two unequal squares joined by a thin passage (not to scale).
The same Robin parameter is imposed on the entire exterior boundary.}
\label{fig:domain}
\end{figure}

We give a self-contained proof of the spectral convergence needed
here, using variational domain perturbation methods; see
\cite{DancerDaners,Daners} for background. We first exclude low
passage eigenvalues and control the cutoff energy at the attachments.
We then use concavity to obtain convergence of the parameter
derivatives on compact intervals.
\begin{proposition}\label{prop:convergence}
For every fixed $\alpha>0$ and every fixed $k\ge1$,
\begin{equation}\label{eq:convergence}
 \lambda_k(\Omega_\eps;\alpha)\longrightarrow\lambda_k(D;\alpha)
 \,\,\,(\eps\downarrow0).
\end{equation}
Moreover, for every compact interval $I\subset(0,+\infty)$,
the first two eigenvalues of $\Omega_\eps$ are simple in a
neighborhood of $I$ for all sufficiently small $\eps$, and
\[
 \lambda_j(\Omega_\eps;\cdot)\longrightarrow\lambda_j(D;\cdot)
 \quad\text{in }C^1(I),\qquad j=1,2.
\]
\end{proposition}

\begin{proof}
{In Steps 1--3, $\alpha>0$ and $k\ge1$ are fixed.}

\emph{Step 1. A lower bound obtained by separating the three pieces.}
The interfaces are
\[
 \Gamma_{1,\eps}=\{1\}\times(-\eps,\eps)^{n-1},\,\,\,
 \Gamma_{2,\eps}=\{2\}\times(-\eps,\eps)^{n-1}.
\]
On $Q_i$, let $\mathfrak a_{i,\eps}$ be the form in
Lemma~\ref{lem:window}, with $\Gamma_\eps=\Gamma_{i,\eps}$.
On the passage, let
\[
 \mathfrak t_\eps[w]
 =\int_{T_\eps}\abs{\nabla w}^2\dd x
  +\alpha\int_{\partial_{\mathrm{lat}}T_\eps}
        \abs[\big]{\Tr w}^2\dd S,
 \,\,\, w\in H^1(T_\eps),
\]
where $\partial_{\mathrm{lat}}T_\eps$ consists of the lateral faces,
not the two ends. Consider the form
\[
 \mathfrak b_\eps[u_1,w,u_2]
 =\mathfrak a_{1,\eps}[u_1]+\mathfrak t_\eps[w]
  +\mathfrak a_{2,\eps}[u_2]
\]
on the larger space
\[
 \mathcal H_\eps
 =H^1(Q_1)\oplus H^1(T_\eps)\oplus H^1(Q_2).
\]
This means that the functions on the three pieces are allowed to
be independent. In boundary-condition language, the cut faces have
Neumann conditions on both sides, while the exterior faces keep
their Robin conditions.

If $u\in H^1(\Omega_\eps)$, its restrictions define an element
$Ju=(u|_{Q_1},u|_{T_\eps},u|_{Q_2})$ of $\mathcal H_\eps$.
The restrictions have matching traces at the interfaces, and
\[
 \begin{aligned}
 \norm{Ju}_{L^2(Q_1)\oplus L^2(T_\eps)\oplus L^2(Q_2)}^2
 &=\norm{u}_{L^2(\Omega_\eps)}^2,\\
 \mathfrak b_\eps[Ju]&=\mathfrak a_{\Omega_\eps,\alpha}[u].
 \end{aligned}
\]
The second equality holds because the interfaces are interior to
$\Omega_\eps$ and neither form includes a boundary integral there.
Thus the Rayleigh quotient is unchanged on the original space,
while allowing all of $\mathcal H_\eps$ enlarges the set of admissible
subspaces in the min--max principle. If $\mu_{k,\eps}$ denotes the
$k$th eigenvalue of $\mathfrak b_\eps$, then
\[
 \mu_{k,\eps}\le\lambda_k(\Omega_\eps;\alpha).
\]

Let $\nu_{j,i,\eps}$ be the eigenvalues on $Q_i$ with the boundary
term removed on $\Gamma_{i,\eps}$. Since
$\abs{\Gamma_{i,\eps}}_{n-1}=(2\eps)^{n-1}\to0$,
Lemma~\ref{lem:window} gives
\[
 \nu_{j,i,\eps}\longrightarrow\lambda_j(Q_i;\alpha)
 \,\,\,(i=1,2)
\]
for each fixed $j$. On $T_\eps$, the longitudinal Neumann eigenvalues
are $(m\pi)^2$, $m=0,1,2,\ldots$, because its length is one. Each
transverse interval has first Robin eigenvalue
$\eps^{-2}p(\alpha\eps)$. Separation of variables therefore gives
the lowest passage eigenvalue
\begin{equation}\label{eq:tube}
 \begin{aligned}
 \tau_{1,\eps}
 &=0+\sum_{j=1}^{n-1}\eps^{-2}p(\alpha\eps)\\
 &=(n-1)\eps^{-2}p(\alpha\eps)
 \sim\frac{(n-1)\alpha}{\eps}
 \longrightarrow+\infty.
 \end{aligned}
\end{equation}
Both $n\ge2$ and $\alpha>0$ are used in the last limit.
The factor $1/\eps$ expresses the growth of lateral boundary area
relative to passage volume. Thus the positive Robin boundary term,
not the small volume alone, excludes bounded passage eigenvalues.

Let $\widehat\nu_{k,\eps}$ be the $k$th eigenvalue of the two cube
problems together. It is the $k$th smallest entry of
$\{\nu_{j,i,\eps}:1\le j\le k,\ i=1,2\}$, counted with
multiplicity. Changing each entry by at most $\delta$ changes the
$k$th smallest entry by at most $\delta$. Hence
\[
 \abs[\big]{\widehat\nu_{k,\eps}-\lambda_k(D;\alpha)}
 \le\max_{\substack{1\le j\le k\\ i=1,2}}
       \abs[\big]{\nu_{j,i,\eps}-\lambda_j(Q_i;\alpha)}
 \longrightarrow0.
\]
By \eqref{eq:tube}, all passage eigenvalues exceed
$\widehat\nu_{k,\eps}$ for small $\eps$, so
$\mu_{k,\eps}=\widehat\nu_{k,\eps}$. The variational bound
$\lambda_k(\Omega_\eps;\alpha)\ge\mu_{k,\eps}$ now gives
\begin{equation}\label{eq:liminf}
 \liminf_{\eps\downarrow0}\lambda_k(\Omega_\eps;\alpha)
 \ge\lim_{\eps\downarrow0}\mu_{k,\eps}
 =\lambda_k(D;\alpha).
\end{equation}

\emph{Step 2. Cutoff functions near the attachment points.}
Let $P_1=(1,0,\ldots,0)$ and $P_2=(2,0,\ldots,0)$.
We construct Lipschitz functions $\chi_{i,\eps}:\R^n\to[0,1]$
which vanish in a neighborhood of the whole interface
$\Gamma_{i,\eps}$ and satisfy
\begin{equation}\label{eq:cutoffprops}
 \abs[\big]{\supp(1-\chi_{i,\eps})}\longrightarrow0,
 \,\,\,
 \int_{\R^n}\abs{\nabla\chi_{i,\eps}}^2\dd x\longrightarrow0.
\end{equation}
Here and below, $\abs{\cdot}$ applied to a subset of $\R^n$ denotes
its $n$-dimensional Lebesgue measure.

For $n\ge3$, put $\rho_\eps=2\sqrt n\,\eps$ and define, with
$r=\abs{x-P_i}$,
\[
 \chi_{i,\eps}(x)=
 \begin{cases}
 0,&r\le\rho_\eps,\\[2pt]
 \displaystyle\frac{r-\rho_\eps}{\rho_\eps},
    &\rho_\eps<r<2\rho_\eps,\\[6pt]
 1,&r\ge2\rho_\eps.
 \end{cases}
\]
Every point of $\overline{\Gamma_{i,\eps}}$ lies at distance at most
$\sqrt{n-1}\,\eps<\rho_\eps$ from $P_i$, so the required vanishing
holds. If $\omega_n$ denotes the volume of the unit ball, then
\[
 \abs[\big]{\supp(1-\chi_{i,\eps})}
 \le\omega_n(2\rho_\eps)^n\longrightarrow0.
\]
The gradient has magnitude $1/\rho_\eps$ in the annulus and is
zero elsewhere, apart from sets of measure zero. Hence
\[
 \begin{aligned}
 \int_{\R^n}\abs{\nabla\chi_{i,\eps}}^2\dd x
 &=\frac{n\omega_n}{\rho_\eps^2}
      \int_{\rho_\eps}^{2\rho_\eps}r^{n-1}\dd r\\
 &=\omega_n(2^n-1)\rho_\eps^{n-2}
 \longrightarrow0.
 \end{aligned}
\]
A cutoff changing linearly between these two radii would have
nonvanishing energy in dimension two, so a different choice is needed
there.

For $n=2$, let $r_\eps=2\eps$ and $\rho_\eps=\sqrt\eps$.
For $\eps<1/4$, $r_\eps<\rho_\eps$. Define
\[
 \chi_{i,\eps}(x)=
 \begin{cases}
 0,&\abs{x-P_i}\le r_\eps,\\[2pt]
 \displaystyle\frac{\log\bigl(\abs{x-P_i}/r_\eps\bigr)}
 {\log(\rho_\eps/r_\eps)},
   &r_\eps<\abs{x-P_i}<\rho_\eps,\\[6pt]
 1,&\abs{x-P_i}\ge\rho_\eps.
 \end{cases}
\]
The whole interface lies inside the disk of radius $\eps<r_\eps$.
Also,
\[
 \abs[\big]{\supp(1-\chi_{i,\eps})}
 \le\pi\rho_\eps^2=\pi\eps\longrightarrow0.
\]
In the annulus, the radial derivative is
$\bigl(r\log(\rho_\eps/r_\eps)\bigr)^{-1}$. Polar coordinates give
\[
 \begin{aligned}
 \int_{\R^2}\abs{\nabla\chi_{i,\eps}}^2\dd x
 &=\frac{2\pi}{\bigl[\log(\rho_\eps/r_\eps)\bigr]^2}
      \int_{r_\eps}^{\rho_\eps}\frac{\dd r}{r}\\
 &=\frac{2\pi}{\log(\rho_\eps/r_\eps)}\longrightarrow0,
 \end{aligned}
\]
since
\[
 \log(\rho_\eps/r_\eps)
 =\frac12\log(1/\eps)-\log2\longrightarrow+\infty.
\]
This proves both assertions in \eqref{eq:cutoffprops} in every
dimension under consideration.

\emph{Step 3. Trial functions and the upper bound.}
Choose real orthonormal eigenfunctions $\phi_1,\ldots,\phi_k$ for
the first $k$ eigenvalues on $D$, each supported on one cube.
Such a choice is possible even if the two cube spectra have common
eigenvalues. On each cube, the eigenfunctions may be chosen as
products of interval eigenfunctions. Thus $\phi_\ell$ and
$\nabla\phi_\ell$ are bounded on their supporting cube.

If $\phi_\ell$ is supported on $Q_i$, define
\[
 v_{\ell,\eps}(x)=
 \begin{cases}
 \chi_{i,\eps}(x)\phi_\ell(x),&x\in Q_i,\\
 0,&x\in\Omega_\eps\setminus Q_i.
 \end{cases}
\]
The cutoff vanishes in a neighborhood of the attachment face.
The pieces therefore have matching zero traces there, and their
union belongs to $H^1(\Omega_\eps)$.

We spell out the convergence on the fixed cube. Put
$S_{i,\eps}=Q_i\cap\supp(1-\chi_{i,\eps})$ and choose a constant
$M_\ell$ such that
$\abs{\phi_\ell}+\abs{\nabla\phi_\ell}\le M_\ell$ on $Q_i$.
Since $0\le\chi_{i,\eps}\le1$,
\[
 \norm[\big]{(\chi_{i,\eps}-1)\phi_\ell}_{L^2(Q_i)}^2
 \le M_\ell^2\abs{S_{i,\eps}}\longrightarrow0.
\]
The product rule gives
\[
 \nabla\bigl((\chi_{i,\eps}-1)\phi_\ell\bigr)
 =(\chi_{i,\eps}-1)\nabla\phi_\ell
   +\phi_\ell\nabla\chi_{i,\eps}.
\]
Using $\abs{a+b}^2\le2\abs{a}^2+2\abs{b}^2$, we obtain
\[
 \begin{aligned}
 \norm[\Big]{\nabla\bigl((\chi_{i,\eps}-1)\phi_\ell\bigr)}_{L^2(Q_i)}^2
 &\le2M_\ell^2\abs{S_{i,\eps}}
   +2M_\ell^2\int_{Q_i}\abs{\nabla\chi_{i,\eps}}^2\dd x\\
 &\longrightarrow0
 \end{aligned}
\]
by \eqref{eq:cutoffprops}. Hence
$\chi_{i,\eps}\phi_\ell\to\phi_\ell$ strongly in $H^1(Q_i)$.
The continuity of the trace on this fixed cube also yields
\[
 \norm[\Big]{\Tr(\chi_{i,\eps}\phi_\ell)-\Tr\phi_\ell}_{L^2(\partial Q_i)}
 \le C_i\norm[\big]{(\chi_{i,\eps}-1)\phi_\ell}_{H^1(Q_i)}
 \longrightarrow0.
\]
In particular, the boundary integrals converge. For example, the
absolute difference of the two squared trace integrals is at most
\[
 \begin{aligned}
 \norm[\Big]{\Tr(\chi_{i,\eps}\phi_\ell)-\Tr\phi_\ell}_{L^2(\partial Q_i)}
\times\left(
   \norm[\big]{\Tr(\chi_{i,\eps}\phi_\ell)}_{L^2(\partial Q_i)}
   +\norm[\big]{\Tr\phi_\ell}_{L^2(\partial Q_i)}\right),
 \end{aligned}
\]
which tends to zero. All trace estimates here are on the fixed
cubes; no uniform trace estimate on $\Omega_\eps$ is needed.

Let $w_{\ell,\eps}=v_{\ell,\eps}|_D$. The preceding estimates say
that $w_{\ell,\eps}\to\phi_\ell$ in $H^1(D)$, where the squared norm is the
sum of the squared $H^1$ norms on the two components.
Because $v_{\ell,\eps}$ is zero in the passage and its cube trace
is zero on the attachment face, we have the exact identities
\[
 \begin{aligned}
 \langle v_{\ell,\eps},v_{m,\eps}\rangle_{L^2(\Omega_\eps)}
 &=\langle w_{\ell,\eps},w_{m,\eps}\rangle_{L^2(D)},\\
 \mathfrak a_{\Omega_\eps,\alpha}
       (v_{\ell,\eps},v_{m,\eps})
 &=\mathfrak a_{D,\alpha}(w_{\ell,\eps},w_{m,\eps}).
 \end{aligned}
\]
Here $\mathfrak a(\cdot,\cdot)$ denotes the symmetric bilinear form
associated with the quadratic form. In the second identity the
Robin boundary term on each attachment face is zero on both sides;
thus replacing the exterior cube faces by the full cube boundaries
introduces no extra term. Strong $H^1$ convergence and continuity
of the trace now give
\[
 \begin{aligned}
 M_{\ell m,\eps}
 &:=\langle v_{\ell,\eps},v_{m,\eps}\rangle_{L^2(\Omega_\eps)}
   \longrightarrow\delta_{\ell m},\\
 A_{\ell m,\eps}
 &:=\mathfrak a_{\Omega_\eps,\alpha}
       (v_{\ell,\eps},v_{m,\eps})
   \longrightarrow\lambda_\ell(D;\alpha)\delta_{\ell m}.
 \end{aligned}
\]
The second limit also uses the eigenfunction identity
$\mathfrak a_{D,\alpha}(\phi_\ell,\phi_m)
=\lambda_\ell(D;\alpha)\langle\phi_\ell,\phi_m\rangle$.

To obtain a bound for all trial functions at once, write
$M_\eps=(M_{\ell m,\eps})$, $A_\eps=(A_{\ell m,\eps})$, and
\[
 \Lambda=\diag\bigl(\lambda_1(D;\alpha),\ldots,
                       \lambda_k(D;\alpha)\bigr).
\]
Set
\[
 \theta_\eps=k\max_{1\le\ell,m\le k}
 \max\left\{
   \abs{M_{\ell m,\eps}-\delta_{\ell m}},
   \abs{A_{\ell m,\eps}-\Lambda_{\ell m}}\right\}
 \longrightarrow0.
\]
For $b\in\R^k$, the inequality
$\bigl(\sum_{\ell=1}^k\abs{b_\ell}\bigr)^2\le k\abs{b}^2$
gives
\[
 \begin{aligned}
 \norm*{\sum_{\ell=1}^k b_\ell v_{\ell,\eps}}_{L^2(\Omega_\eps)}^2
 &=b^{\mathsf T}M_\eps b
 \ge(1-\theta_\eps)\abs{b}^2,\\
 \mathfrak a_{\Omega_\eps,\alpha}
   \left[\sum_{\ell=1}^k b_\ell v_{\ell,\eps}\right]
 &=b^{\mathsf T}A_\eps b
 \le\bigl(\lambda_k(D;\alpha)+\theta_\eps\bigr)\abs{b}^2.
 \end{aligned}
\]
For small $\eps$, $\theta_\eps<1$, so the trial functions are
linearly independent. Applying the min--max principle on
$V_\eps=\spanop\{v_{1,\eps},\ldots,v_{k,\eps}\}$ gives
\[
 \lambda_k(\Omega_\eps;\alpha)
 \le\max_{v\in V_\eps\setminus\{0\}}
       \frac{\mathfrak a_{\Omega_\eps,\alpha}[v]}
            {\norm{v}_{L^2(\Omega_\eps)}^2}
 \le\frac{\lambda_k(D;\alpha)+\theta_\eps}{1-\theta_\eps}.
\]
Taking the upper limit and using \eqref{eq:liminf} proves
\eqref{eq:convergence}.

\emph{Step 4. Uniform convergence and parameter derivatives.}
Fix a compact interval $I\subset(0,+\infty)$ and choose a larger
compact interval $J\subset(0,+\infty)$ with $I\subset\interior J$.
Write
\[
 \lambda_{j,\eps}(\alpha)=\lambda_j(\Omega_\eps;\alpha),
 \qquad \ell_j(\alpha)=\lambda_j(D;\alpha).
\]
For $j=1,2,3$, the convergence in \eqref{eq:convergence} is uniform
on $J$. Indeed, the min--max principle makes $\lambda_{j,\eps}$
nondecreasing in $\alpha$, and the ordered cube eigenvalues
$\ell_j$ are continuous. Given $\delta>0$, partition $J$ into
finitely many intervals on each of which $\ell_j$ changes by less
than $\delta$. Once the errors at all partition points are less
than $\delta$, monotonicity bounds the error at every intermediate
point by $2\delta$.

By Lemma~\ref{lem:disconnected}, $\ell_1<\ell_2<\ell_3$ on $J$.
Continuity and compactness give
\[
 \gamma:=\min_{\alpha\in J}
 \min\bigl\{\ell_2(\alpha)-\ell_1(\alpha),
            \ell_3(\alpha)-\ell_2(\alpha)\bigr\}>0.
\]
Uniform convergence therefore preserves both gaps, with lower
bound $\gamma/2$, for sufficiently small $\eps$. In particular,
$\lambda_{1,\eps}$ and $\lambda_{2,\eps}$ are simple throughout $J$.
For each fixed such $\eps$, the Robin form has the fixed domain
$H^1(\Omega_\eps)$ and depends affinely on $\alpha$; its boundary
term is bounded with respect to the $H^1$ norm by the trace theorem.
Analytic perturbation theory for forms thus makes these simple
eigenvalues analytic on $\interior J$; see
\cite[Chapter~VII, Sections~3--4]{Kato}.

Set $F_\eps=\lambda_{1,\eps}$ and
$S_\eps=\lambda_{1,\eps}+\lambda_{2,\eps}$.
The Rayleigh principle shows that $F_\eps$ is concave in $\alpha$.
The variational formula for the sum of the first two eigenvalues is
\[
 S_\eps(\alpha)
 =\inf_{\substack{u_1,u_2\in H^1(\Omega_\eps)\\
       \langle u_i,u_j\rangle_{L^2(\Omega_\eps)}=\delta_{ij}
       \ (i,j=1,2)}}
 \bigl(\mathfrak a_{\Omega_\eps,\alpha}[u_1]
       +\mathfrak a_{\Omega_\eps,\alpha}[u_2]\bigr).
\]
Indeed, fix $\eps$ and $\alpha$, and let
$(e_m)_{m\ge1}$ be an $L^2(\Omega_\eps)$-orthonormal eigenbasis with
corresponding eigenvalues $\mu_m=\lambda_m(\Omega_\eps;\alpha)$.
For any admissible pair $(u_1,u_2)$, write
\[
 c_{im}=\langle u_i,e_m\rangle_{L^2(\Omega_\eps)},
 \qquad p_m=\abs{c_{1m}}^2+\abs{c_{2m}}^2.
\]
Parseval's identity gives $\sum_{m\ge1}p_m=2$. Bessel's inequality,
applied to $e_m$ and the orthonormal pair $(u_1,u_2)$, gives
$0\le p_m\le\norm{e_m}_{L^2(\Omega_\eps)}^2=1$.
The spectral representation of the quadratic form on its domain
$H^1(\Omega_\eps)$ therefore yields
\[
 \begin{aligned}
 \mathfrak a_{\Omega_\eps,\alpha}[u_1]
 +\mathfrak a_{\Omega_\eps,\alpha}[u_2]
 &=\sum_{m\ge1}\mu_m p_m\\
 &\ge\mu_1p_1+\mu_2\sum_{m\ge2}p_m\\
 &=\mu_1p_1+\mu_2(2-p_1)\\
 &=\mu_1+\mu_2+(\mu_2-\mu_1)(1-p_1)\\
 &\ge\mu_1+\mu_2.
 \end{aligned}
\]
Taking $u_1=e_1$ and $u_2=e_2$ gives equality, proving the formula.
Its
right-hand side is an infimum of affine functions of $\alpha$,
so $S_\eps$ is concave as well. Their uniform limits on $J$ are
$F=\ell_1$ and $S=\ell_1+\ell_2$, which are smooth by
\eqref{eq:ordering}.

We use an elementary consequence of concavity. If a differentiable
concave function $f_\eps$ converges uniformly on $J$ to a function
$f\in C^1(J)$, then, for $\alpha\in I$ and sufficiently small $h>0$,
\[
 \frac{f_\eps(\alpha+h)-f_\eps(\alpha)}h
 \le f_\eps'(\alpha)
 \le\frac{f_\eps(\alpha)-f_\eps(\alpha-h)}h.
\]
More explicitly, take $0<h<\operatorname{dist}(I,\partial J)$,
so that $\alpha\pm h\in J$ for every $\alpha\in I$, and put
$E_\eps=\sup_{\beta\in J}\abs{f_\eps(\beta)-f(\beta)}$.
Each difference quotient contains two function values, so
\[
 \begin{aligned}
 \frac{f_\eps(\alpha+h)-f_\eps(\alpha)}h
 &\ge\frac{f(\alpha+h)-f(\alpha)}h-\frac{2E_\eps}h,\\
 \frac{f_\eps(\alpha)-f_\eps(\alpha-h)}h
 &\le\frac{f(\alpha)-f(\alpha-h)}h+\frac{2E_\eps}h.
 \end{aligned}
\]
The fundamental theorem of calculus also gives
\[
 \begin{aligned}
 \frac{f(\alpha+h)-f(\alpha)}h
 &=\frac1h\int_0^h f'(\alpha+t)\dd t
 \ge f'(\alpha)-\omega_{f'}(h),\\
 \frac{f(\alpha)-f(\alpha-h)}h
 &=\frac1h\int_0^h f'(\alpha-t)\dd t
 \le f'(\alpha)+\omega_{f'}(h).
 \end{aligned}
\]
Here $\omega_{f'}$ is defined below; it tends to zero because $f'$
is uniformly continuous on the compact interval $J$.
Combining these estimates with the concavity bounds yields
\[
 -\omega_{f'}(h)-\frac{2E_\eps}h
 \le f_\eps'(\alpha)-f'(\alpha)
 \le\omega_{f'}(h)+\frac{2E_\eps}h.
\]
Comparing these slopes with those of $f$ gives
\[
 \sup_{\alpha\in I}\abs{f_\eps'(\alpha)-f'(\alpha)}
 \le \omega_{f'}(h)
    +\frac2h\sup_{\alpha\in J}\abs{f_\eps(\alpha)-f(\alpha)},
\]
where
\[
 \omega_{f'}(h)
 =\sup\bigl\{\abs{f'(x)-f'(y)}:
                x,y\in J,\ \abs{x-y}\le h\bigr\}
 \longrightarrow0\quad(h\downarrow0).
\]
First letting $\eps\downarrow0$ and then $h\downarrow0$ proves
uniform convergence of the derivatives on $I$. Apply this argument
to $F_\eps$ and $S_\eps$. Since
$\lambda_{2,\eps}=S_\eps-F_\eps$, both eigenvalues converge in
$C^1(I)$, as claimed.
\end{proof}

\begin{proof}[Proof of Theorem~\ref{thm:counterexample}]
Choose $s$ as in \eqref{eq:srange} and consider the domains
$\Omega_\eps$ in \eqref{eq:domain}. By
Proposition~\ref{prop:convergence} and the positivity of
$\lambda_1(D;\alpha)$,
\[
 \calR_{\Omega_\eps}\longrightarrow R_s
 \quad\text{in }C^1\bigl([\alpha_-,\alpha_+]\bigr).
\]
Lemma~\ref{lem:disconnected} and continuity give
\[
 m:=\min_{\alpha\in[\alpha_-,\alpha_+]}R_s'(\alpha)>0.
\]
Hence, for all sufficiently small $\eps$,
\[
 \calR_{\Omega_\eps}'(\alpha)>\frac m2
 \quad\text{for every }\alpha\in[\alpha_-,\alpha_+].
\]
Choose one such $\eps$ and set $\Omega=\Omega_\eps$. This gives the
asserted strict increase, and integration yields \eqref{eq:reversal}.
The required geometry follows from \eqref{eq:domain}.
\end{proof}

\begin{remark}\label{rem:nonmonotone}
For each fixed $\eps>0$, connectedness gives
$\calR_{\Omega_\eps}(\alpha)\to+\infty$ as
$\alpha\downarrow0$~\cite{Laugesen}. Together with
\eqref{eq:reversal}, this rules out monotonicity in either direction
on $(0,+\infty)$ for the domain chosen in
Theorem~\ref{thm:counterexample}. By Proposition~\ref{prop:convergence} and
Lemma~\ref{lem:disconnected}, the two limits do not commute:
\[
 \lim_{\alpha\downarrow0}\,
      \lim_{\eps\downarrow0}\calR_{\Omega_\eps}(\alpha)=\frac1s,
 \,\,\,
 \lim_{\eps\downarrow0}\,
      \lim_{\alpha\downarrow0}\calR_{\Omega_\eps}(\alpha)=+\infty.
\]
The passage width is chosen for the prescribed compact interval;
no fixed width can give strict increase on all of $(0,+\infty)$.
\end{remark}

\begin{remark}
In the planar construction, one may take $s=9/10$,
$\alpha_-=1$ and $\alpha_+=2$.
The proof gives a positive threshold for the passage width below
which $\calR_{\Omega_\eps}'(\alpha)>0$ throughout $[1,2]$, and hence
\eqref{eq:reversal} holds.
\end{remark}

For negative parameters, we distinguish a fixed domain from a
shrinking passage. On a fixed connected Lipschitz domain $\Omega$,
the constant trial function gives
\[
 \lambda_1(\Omega;\alpha)
 \le\alpha\frac{\abs{\partial\Omega}_{n-1}}{\abs{\Omega}}<0,
 \,\,\,
 \calR_\Omega(\alpha)<1.
\]
Write $L(\alpha)=\lambda_1(\Omega;\alpha)$. For $a<b<0$, let
$u_b$ be an $L^2(\Omega)$-normalized first eigenfunction at parameter
$b$. The variational principle gives
\[
 \begin{aligned}
 L(a)&\le\mathfrak a_{\Omega,a}[u_b]\\
 &=L(b)-(b-a)\int_{\partial\Omega}\abs[\big]{\Tr u_b}^2\dd S\\
 &<L(b).
 \end{aligned}
\]
Indeed, a zero trace would imply
$L(b)=\int_\Omega\abs{\nabla u_b}^2\dd x\ge0$, contrary to $L(b)<0$.
Thus $L$ strictly increases. The min--max principle also gives
$\lambda_2(\Omega;a)\le\lambda_2(\Omega;b)$. Whenever
$\lambda_2(\Omega;a)>0$, it follows that
\[
 \frac{\lambda_2(\Omega;b)}{-L(b)}
 \ge\frac{\lambda_2(\Omega;a)}{-L(b)}
 >\frac{\lambda_2(\Omega;a)}{-L(a)}.
\]
Hence $\calR_\Omega(b)<\calR_\Omega(a)$. This proves strict decrease wherever
$\alpha<0$ and $\lambda_2>0$, in particular near $0^-$. The Neumann
limits $\lambda_1\to0^-$ and $\lambda_2\to\lambda_2(\Omega;0)>0$
then give $\calR_\Omega\to-\infty$~\cite{Laugesen}.
For the domains $\Omega_\eps$, every fixed eigenvalue tends instead
to $-\infty$ as the passage shrinks. To see this, fix $\alpha<0$
and $k\ge1$, and take
\[
 f\in W_k:=\spanop\bigl\{\sin\bigl(m\pi(x_1-1)\bigr):
                         1\le m\le k\bigr\}\subset H^1_0(1,2).
\]
Writing $x=(x_1,x')$, set $u(x_1,x')=f(x_1)$ in $T_\eps$ and
extend $u$ by zero to the two cubes. The zero end values of $f$ ensure
$u\in H^1(\Omega_\eps)$. Direct integration gives
\[
 \begin{aligned}
 \norm{u}_{L^2(\Omega_\eps)}^2
 &=(2\eps)^{n-1}\int_1^2\abs{f}^2\dd x_1,\\
 \mathfrak a_{\Omega_\eps,\alpha}[u]
 &=(2\eps)^{n-1}\int_1^2\abs{f'}^2\dd x_1
   +2\alpha(n-1)(2\eps)^{n-2}\int_1^2\abs{f}^2\dd x_1.
 \end{aligned}
\]
For $f\ne0$, orthogonality of the sine functions yields
\[
 \frac{\mathfrak a_{\Omega_\eps,\alpha}[u]}
      {\norm{u}_{L^2(\Omega_\eps)}^2}
 =\frac{\int_1^2\abs{f'}^2\dd x_1}{\int_1^2\abs{f}^2\dd x_1}
   +\frac{(n-1)\alpha}{\eps}
 \le(k\pi)^2+\frac{(n-1)\alpha}{\eps}.
\]
The min--max principle on this $k$-dimensional trial space gives
\begin{equation}\label{eq:negative-collapse}
 \lambda_k(\Omega_\eps;\alpha)
 \le(k\pi)^2+\frac{(n-1)\alpha}{\eps}
 \longrightarrow-\infty\,\,\,(\eps\downarrow0).
\end{equation}
Here $\alpha<0$ is fixed and $\eps\downarrow0$. Comparing
\eqref{eq:tube} with \eqref{eq:negative-collapse}, the sign of the
Robin boundary term explains why the positive-parameter spectral
limit fails: negative parameters create arbitrarily low passage
modes.
At $\alpha=0$, the first eigenvalue is zero and the ratio is
undefined.

The constructed domains are not convex, so
Theorem~\ref{thm:counterexample} leaves the convex case open.
Theorem~\ref{thm:boxes} and the ball result in~\cite{DaiSun}
suggest the following question.
\begin{question}
If $\Omega\subset\R^n$ is a bounded convex domain, is
$\alpha\mapsto\calR_\Omega(\alpha)$ strictly decreasing on
$(0,+\infty)$?
\end{question}

On boxes, the proof works because the first eigenvalue is a sum of
interval eigenvalues and the gap is determined by a longest edge.
For a general convex domain, these separation-of-variables formulas
are unavailable. The remaining issue is whether convexity alone
can provide a comparison of the relative growth of the first two
eigenvalues.
\\ \\
\textbf{The conflicts of interest statement and Data Availability statement.}
\bigskip\\
\indent There is not any conflict of interest.
Data sharing not applicable to this article as no datasets were generated or analysed during the current study.
\\\\
\textbf{Declaration of AI-assisted writing.}
\bigskip\\
\indent During the preparation of this work, the authors used ChatGPT and DeepSeek for language polishing, grammar checking and refining the exposition of certain technical arguments. After using these tools, the authors carefully reviewed and edited the content as needed and take full responsibility for the final version of this manuscript.

\end{document}